\documentclass{ieja}
\usepackage{amsmath,amsthm,amssymb,amsfonts,amscd, array}
\usepackage{enumerate}
\usepackage{graphicx}
\usepackage{amsmath}
\newtheorem{theorem}{Theorem}

\newtheorem{corollary}[theorem]{Corollary}

\newtheorem{lemma}[theorem]{Lemma}

\newtheorem{proposition}[theorem]{Proposition}

\newtheorem{remark}[theorem]{Remark}

\theoremstyle{plain}

\theoremstyle{remark}

\begin{document}
\title{On two questions of Nicholson}
\author{Feroz Siddique}

\address{
	{\bf Feroz Siddique} \\
	University of Wisconsin-Eau Claire \\  Barron County, WI-54868, USA.}
\email{E-mail: siddiquf@uwec.edu}

\begin{abstract}
We show that a ring R has stable range one if and only if every left unit lifts modulo every left principal ideal. We also show that a left quasi-morphic ring has stable range one if and only if it is left uniquely generated. Thus we answer in the affirmative the two questions raised  by W. K. Nicholson.
 \\[+2mm]
\subjclassname{ 16U60, 16D50}\\
\end{abstract}

\maketitle
\section{ Introduction}
\bigskip 
\noindent All our rings are associative  with identity and all our modules are left modules, unless specifically mentioned. Let  $U(R)$ denote the group of all units of $R$. \\
In a ring $R$,  an element $x$  is a left (right) unit if there exists $y \in R$ such that $yx =  1(xy = 1)$. We say that a left unit $x$ lifts modulo any left ideal I if whenever  $yx - 1 \in I$ for some $y \in R$, there exists a left unit $u$ in $R$ such that $x-u \in I$.  Similarly we say that a right unit $x'$ lifts modulo any left ideal I if whenever  $x'y' - 1 \in I$ for some $y' \in R$, there exists a right unit $v$ in $R$ such that $x'-v \in I$. \\

Recall that a ring $R$ is called directly finite (or Dedekind finite or Von Neumann finite) if for any two elements $a,b \in R$, $ab = 1 \implies ba = 1$. A module $M$ is called directly finite if $M \cong M \bigoplus N \implies  N=0$. Furthermore, $M$ is a Dedekind finite module if and only if $End(M_R)$ is a Dedekind finite ring. \\

In \cite{Can}, Canfell defines a ring $R$ to be \textit{right uniquely generated} if for any two right principal ideals $aR$ and $bR$ in $R,$ if  $aR = bR$   then  $a =bu$ for some $u \in U(R)$. The notion of \textit{left uniquely generated} ring is defined similarly. This concept of principal left (right) ideals being uniquely generated first appeared in Kaplansky's classic paper \cite{Ka}. In it, he had raised the question of when a  ring R satisfies the property of being right uniquely generated. \\
A ring $R$ is called \textit{left (right) quasi-morphic} if the collection of all left (right) principal  ideals coincides with the collection of all left (right) annihilators in the ring. For example, any Von Neumann regular ring is both left and right quasi-morphic. Left (right) quasi-morphic were first introduced and studied by Camillo and Nicholson as a generalization of left (right) morphic rings in \cite{NiCa}.  \\

In \cite{Ba1}, Bass first introduced the notion of (left) right stable range for a ring (denoted as $sr(R)$). Vasershtein in \cite{Va1} showed that the notion of stable range one is left-right symmetric. A ring $R$ is said to have stable range one if for any $a,b,x \in R, ax + b = 1 \implies a + by = u \in U(R)$ for some $y \in R$. In the literature of rings with stable range one, many equivalent conditions have been to shown to characterize such a class of rings. We would like to mention one below.\\

In \cite{Gd}, Goodearl stated  the following lemma which is essentially due to Vasershtein. 
\begin{lemma} 
	Let $a,b,c$ be elements of a ring $R$, such that $ab + c = 1$. If there exists $x \in R$ such that $a+cx$ is invertible, then there exists $y \in R$ such that $b+yc$ is invertible. 
\end{lemma} 
 The symmetrical analogue of Vasershtein's Lemma also holds true as we see in the following lemma. The proof follows  similarly to that of Lemma 1. 
 \begin{lemma} 
	Let $a,b,c$ be elements of a ring $R$, such that $ab + c = 1$. If there exists $y \in R$ such that $b+yc$ is invertible, then there exists $x \in R$ such that $a+cx$ is also invertible.
\end{lemma}

\begin{proof}
	Let $ab + c = 1$ and choose $y \in R$ such that $b+yc \in U(R)$. \\
	\noindent Set $u = b+yc \in U(R)$. \\
	Let $v = a + cu^{\scriptsize -1}(1-ya)$ and $w = b + (1-ba)y$. \\
	We prove that $v \in U(R)$. \\
	{
		\setlength{\abovedisplayskip}{1.0pt}
		\setlength{\belowdisplayskip}{2.0pt}
		\begin{alignat*}{3}
		\text{Now,} \hspace{20mm}\\
		bv&= ba+bcu^{\scriptsize -1}(1-ya)  \hspace{35mm}  \hspace{20mm} &&\cdots \hspace{10mm} (i)\\
		\text{And,} \hspace{20mm}\\
		yv&= ya + ycu^{\scriptsize -1}(1-ya)\\
		&= ya + (u-b)u^{\scriptsize -1}(1-ya)\\ 
		&= ya + (1-bu^{\scriptsize -1})(1-ya)\\
		&= 1-bu^{\scriptsize -1}(1-ya) \\
		\text{Then,} \hspace{20mm}\\
		(1-ba)yv&= (1-ba) - (1-ba)bu^{\scriptsize -1}(1-ya)\\
		&= (1-ba) - (b-bab)u^{\scriptsize -1}(1-ya)\\
		&= (1-ba) - b(1-ab)u^{\scriptsize -1}(1-ya) \hspace{5mm} &&\cdots \hspace{10mm} (ii) \\
		\end{alignat*}
	}
	Using (i) and (ii) we get, 
	{
		\setlength{\abovedisplayskip}{2.5pt}
		\setlength{\belowdisplayskip}{2.0pt}
		\begin{alignat*}{3}
		\hspace{15mm}	wv&= (b+(1-ba)y)v \\
		&= bv + (1-ba)yv \\
		&= ba+bcu^{\scriptsize -1}(1-ya)  + (1-ba) - b(1-ab)u^{\scriptsize -1}(1-ya)\\
		&= 1 + (bc - b(1-ab))u^{\scriptsize -1}(1-ya) \\
		&= 1
		\end{alignat*}
	}
	
	Now, observe that \\
	{
		\setlength{\abovedisplayskip}{2.0pt}
		\setlength{\belowdisplayskip}{2.0pt}
		\begin{alignat*}{3}
		aw&= a( b + (1-ba)y)\\
		&= ab + (1-ab)ay \\ 
		&= a b + cay \hspace{55mm} \cdots \hspace{10mm} (iii)\\
		\text{Then,} \hspace{15mm}\\
		(1-ya)w &= w - y(aw) \\
		&= w - y(a b + cay ) \\
		&= b + (1-ba)y - yab - ycay \\
		&= b + y(1-ab) - (b+yc)ay \\
		&= b + yc - uay \\
		&= u(1-ay)  \hspace{30mm} \\
		\text{Then, }\hspace{15mm}\\
		\hspace{10mm}cu^{\scriptsize -1}(1-ya)w &= c(1-ay) \hspace{54mm} \cdots \hspace{10mm} (iv)
		\end{alignat*}
	}
	
	Adding (iii) and (iv) we get, 
	{
		\setlength{\abovedisplayskip}{2.0pt}
		\setlength{\belowdisplayskip}{2.0pt}
		\begin{alignat*}{3}
		&aw+ cu^{\scriptsize -1}(1-ya)w &&= ab +cay + c(1-ay) =1
		\end{alignat*}
	}
	Therefore, 	$(a + cu^{\scriptsize -1}(1-ya))w= vw = 1$ and consequently,  \\
	\begin{center}
		$a + cu^{\scriptsize -1}(1-ya) \in U(R)$
	\end{center}
\end{proof}

 The study of rings satisfying stable range conditions and its consequences has generated a lot of interest in the past. 
 Bass showed that every semi-local ring $R$ has stable range one. This was later generalized independently by Fuchs \cite{Fu2}, Kaplansky \cite{Ka2} and Henriksen \cite{He} where they showed that every unit-regular ring has stable range one. 

\bigskip

\noindent  Nicholson \cite{Nicholson} asked the  following two questions:

\bigskip

\begin{enumerate}
\item If every left unit lifts modulo every left ideal, is the ring $R$ directly finite?

\bigskip

\item If a ring $R$ is left quasi morphic and left uniquely generated, does $R$ have stable range one?\\
\end{enumerate}

\noindent \noindent The problem of lifting of units modulo an ideal has been studied on many instances. For example,  in \cite{menmon1} Menal and Moncasi showed that for a  regular right self injective ring of Type III, every unit lifts modulo every 2-sided ideal. Nicholson's first question deal with when every left unit lifts modulo every left ideal. In this paper we answer both the questions affirmatively and in fact,  prove a stronger version of the above questions.\\
 
 %

\section{Main Results}
\noindent We begin with the following lemma.. 

\begin{lemma}
Let R be a ring. Then $R$ is directly finite if any of the following conditions hold.
\begin{enumerate}
\item Every left unit lifts modulo every left principal ideal. \\
\item Every right unit lifts modulo every right principal ideal.\\
\end{enumerate}
\end{lemma}

\begin{proof}
Suppose (1) holds. We claim $R$ is directly finite.  \\
Choose  $a,b \in R$ such that $ab = 1$ and consider the left principal ideal $R(1-ba)$. Now $(ba - 1) \in R(1-ba)$, and therefore, $a$ is a left unit modulo the left principal ideal $R(1-ba)$. Hence, by our hypothesis there exists a left unit $u \in R$ such that 
\begin{center}
$a-u\in R(1-ba)$
\end{center}
Since $u$ is a left unit in $R$, there exists an element $ v \in R$ such that  $vu = 1$.\\ Let $a-u = x(1-ba)$. Multiplying on the right by $b$ we have

			{
	\setlength{\abovedisplayskip}{2.5pt}
	\setlength{\belowdisplayskip}{2.0pt}
	\begin{alignat*}{3}
		&(a - u)b &&= x(1-ba)b \\
		\implies&  ab- ub &&= xb - xbab\\ 
		\implies& 1 - ub &&= 0
\end{alignat*}
}
	
\noindent Hence  $ub = 1$ and  therefore, $u$ being both left and right invertible, is a unit. \\Then, $v = b \in U(R)$ and consequently,  $ ba = 1$.\\
Thus, $R$ is directly finite. \\
The proof traverses in a similar way if we assume (2).\\
\end{proof}

\bigskip 

\noindent We now proceed to strengthen  the above lemma in the next theorem.
\bigskip 

\begin{theorem}
Let R be a ring. Then the following are equivalent:\\
\begin{enumerate}
\item R has stable range one.\\
\item Every left unit lifts modulo every left principal ideal.\\
\item Every right unit lifts modulo every right principal ideal.\\
\end{enumerate}

\end{theorem}

\begin{proof}
$(1) \implies (2):$ Let $R$ have stable range one. Choose $a,b,c \in R$ such that $ab - 1 \in Rc$, that is, $b$ is a left unit modulo the left principal ideal $Rc$. We prove that $b$ lifts to a left unit modulo the left principal ideal $Rc$.\\
Let $x \in R$ such that $ab - 1 = xc$. Then $ab - xc = 1$. \\
Since $R$ has stable range one, we must have $a +(-xc)y \in U(R)$ for some $y\in R$. Then from Lemma  3 there exists $t \in R, u \in U(R)$ such that 
			{
	\setlength{\abovedisplayskip}{2.5pt}
	\setlength{\belowdisplayskip}{2.0pt}
	\begin{alignat*}{3}
	&b + t(-xc )&&= u \\
	\implies& b-u &&= txc
	\end{alignat*}
}
Therefore $b - u \in Rc$ where $u$ is invertible (and hence left invertible) in $R$. \\
Hence, $b$ lifts to a left unit $u$, modulo the left prinicpal ideal $Rc$.\\

$(2) \implies (1)$: Assume every left unit lifts modulo every left principal ideal. \\
We prove that $R$ has stable range one.\\
\noindent Let $a,b,c \in R$ such that $ab + c = 1$. Then $ab-1  \in Rc$ which shows that $b$ is a left unit modulo the left principal ideal $Rc$. Hence, by our hypothesis, there exists a left unit $u \in R$ such that $b-u \in Rc$. \\
Now $R$ being directly finite from Lemma 3 forces every  left unit to be  invertible in $R$. Thus $u \in U(R)$. \\
Choose $s \in R$ such that $b-u = sc$. \\
Then, $ b + (-s)c = u \in U(R)$. Therefore from Lemma 2, we have 
\begin{center}
$a+cx \in U(R)$
\end{center}
Hence  $R$ has stable range one.
\bigskip 

$(1) \Leftrightarrow (3)$: The proof follows similarly.
\end{proof}
 \bigskip

\noindent  Since rings with stable range one are directly finite, Theorem 4 gives a stronger answer to Question 1.

\bigskip


In his classic paper \cite{Ka}, Kaplansky considered several examples where a ring $R$ satisfies the property of being right uniquely generated and where it is not. He remarked that for commutative rings, the property holds true for principal ideal rings and artinian rings. Furthermore, Hartwig \cite{hart} showed that all unit-regular rings are right uniquely generated. For a counterexample in the general case,  the ring of all linear transformations over an infinite dimensional vector space was illustrated by Kaplansky as a case in point. This ring is clearly von Neumann regular but not unit-regular, and consequently not stable range one. 
 Now it is known that if a ring $R$ is von Neumann regular then $R$ has stable range one if and only if it is unit- regular. Reducing the assumptions from von Neumann regular rings to (left) right quasi-morphic rings and unit-regular rings to (left) right uniquely generated rings it is quite natural to check if the ring $R$  still satisfies stable-range one ?\\

We first strengthen Nicholson's assertion and prove that the left (right) uniquely generated property is  equivalent to Bass's stable range  on left (right) quasi-morphic rings.  
We   consider the following proposition. 
\begin{proposition}
		Let $R$ be a ring. If $R$ is left (right) uniquely generated, then $R$ is directly finite. 
\end{proposition}
\noindent But $R$  need not have stable range one. Consider $R= \mathbb Z$. Then R is both left and right uniquely generated but $R$ does not have stable range one. \\
Recall that Camillo and Nicholson \cite{NiCa} defined a ring $R$ to be {\it left quasi-morphic} if the collection of all principal left ideals in $R$ coincides with the collection of all principal left  annihilators in $R$, that is, for every element $a \in R$, there exists $b,c \in R$ such that $Ra = ann_l(b)$ and $ann_l(a) = Rc$. Similarly, it is defined for {\it right quasi-morphic} rings.  Yang \cite{yang} generalized this notion of {\it left (right) quasi-morphic} rings and defined {\it left (right) pseudo-morphic} rings where every left (right) principal ideal is a left (right) annihilator of some element in the ring. \\
In the next theorem we will prove Nicholson's second question in the affirmative by showing the equivalency of left  uniquely generated property with stable range one for  {\it left  pseudo-morphic} rings. \\

\begin{theorem}
	Let $R$ be a ring. If $R$ is left pseudo-morphic, then the following are equivalent :
	\begin{enumerate}
		\item $R$ is left uniquely generated.\\
		\item $R$ has stable range one.\\
	\end{enumerate}
\end{theorem}

\begin{proof}
	$(1) \implies (2)$ :  In view of Theorem 4, it suffices to show that  every left unit lifts modulo every left principal ideal in $R$.\\
	Let $x$ be a left unit modulo the left principal ideal $Ry$ i.e there exists  $z \in R$ such that $zx - 1 \in Ry$. Let $zx - 1 =ky$. We would like to show that there exists a left unit $u$ such that $x-u \in Ry$.\\
	\noindent Since $R$ is  {\it left pseudo-morphic}, there exists $a,b \in R$ such that $Ry = ann_l(a)$ and $R(xa) = ann_l(b)$. Now,
		{
		\setlength{\abovedisplayskip}{2.5pt}
		\setlength{\belowdisplayskip}{2.0pt}
		\begin{alignat*}{3}
		&zx - 1 &&=ky\\
	\implies	&zx - ky &&= 1 \\
	\implies&Rx + Ry &&= R
		\end{alignat*}
	}
	\noindent But for any $r\in R$, since $rxa \in R(xa) = ann_l(b)$ we  have 
			{
		\setlength{\abovedisplayskip}{2.5pt}
		\setlength{\belowdisplayskip}{2.0pt}
		\begin{alignat*}{3}
		rx(ab) &=  (rxa)b \\
		&= 0
		\end{alignat*}
	}
\noindent Therefore, 	$Rx \subseteq ann_l(ab)$. \\

\noindent Similarly, since $ry \in Ry = ann_l(a)$ we have 
			{
	\setlength{\abovedisplayskip}{2.5pt}
	\setlength{\belowdisplayskip}{2.0pt}
	\begin{alignat*}{3}
	ry(ab) &=  ((ry)a)b \\
	 &=  0.b \\
	 &= 0
	\end{alignat*}
}
\noindent Therefore, 	$Ry \subseteq ann_l(ab)$. 

	\noindent Hence, we must have,
				{
		\setlength{\abovedisplayskip}{2.5pt}
		\setlength{\belowdisplayskip}{2.0pt}
		\begin{alignat*}{3}
	&R &&= Rx + Ry \subseteq  ann_l(ab)\\
	 \implies &ab &&= 0 \\
	 \implies &a &&\in ann_l(b)  \\
	\implies  &Ra &&\subseteq ann_l(b)
		\end{alignat*}
	}
But, $ann_l(b) = R(xa) \subseteq Ra$.  \\
Therefore $ann_l(b) = Ra $. \\
But since, $ann_l(b)= R(xa)$ we must have, $R(xa) = Ra$.\\
	\noindent Now  $R$ being left uniquely generated and $R(xa) = Ra$, there exists a unit $u \in R$ such that 
	{
		\setlength{\abovedisplayskip}{2.5pt}
		\setlength{\belowdisplayskip}{2.0pt}
		\begin{alignat*}{3}
		&xa &&= ua\\
		\implies &(x-u)a &&= 0 \\
		\implies &(x-u) &&\in ann_l(a) = Ry 
		\end{alignat*}
	}
Thus, the left unit $x$ lifts to an invertible element $u$ modulo the left principal ideal $Ry$. Hence, from Theorem 4, the ring $R$ has stable range one. \\
	
	\noindent $(2) \implies (1)$: Let $Ra$ and $Rb$ be any two left principal ideals in $R$ such that $Ra = Rb$. Then there exists $x,y \in R$ such that $a = xb$ and  $b = ya$. Therefore, 
		{
		\setlength{\abovedisplayskip}{2.5pt}
		\setlength{\belowdisplayskip}{2.0pt}
		\begin{alignat*}{3}
		b &= yxb\\
		\implies (1-yx)b &= 0 \\
		 \implies (1-yx) &\in ann_l(b)
		\end{alignat*}
	}
	\noindent Now, $yx + (1-yx) = 1$. Since $R$ is assumed to have stable range one, there exists an $s  \in R$ such that $y + (1-yx)s \in U(R)$. Then, by Lemma 1,  there exists $t\in R$ such that 
\begin{center}
$x + t(1-yx) = u\in U(R)$
\end{center}
 Since $(1-yx) \in ann_l(b)$ we have 
 		{
 	\setlength{\abovedisplayskip}{2.5pt}
 	\setlength{\belowdisplayskip}{2.0pt}
 	\begin{alignat*}{3}
 	(x + t(1-yx))b &= ub\\
 	\implies xb &= ub 
 	\end{alignat*}
 }
where $xb = a$. Thus $a = ub$ for some $u \in U(R)$ which completes the proof.\\
\end{proof}
\bigskip
\begin{remark}
	Note that $(2) \implies (1)$ does not require $R$ to be {\it left pseudo-morphic}. Thus a ring R with stable range one is both left and right uniquely generated.  
\end{remark}

\begin{theorem}
	Let $R$ be a ring. If $R$ is right pseudo-morphic, then the following are equivalent :
\begin{enumerate}
	\item $R$ is right uniquely generated.\\
	\item $R$ has stable range one.\\
\end{enumerate}
\end{theorem}

\begin{proof}
	The proof follows similarly as above. 
	\end{proof}

\bigskip 
  Bass showed that every semi-local ring $R$ has stable range one. This was later generalized independently by Fuchs \cite{Fu2}, Kaplansky \cite{Ka2} and Henriksen \cite{He} where they showed that every unit-regular ring has stable range one.  In \cite{goodearl}, there is a proof attributed to Kaplansky, that a von Neumann regular ring R is unit regular if and only it has stable range one. Later Camillo and Yu \cite{camyu} extended this result to exchange rings and proved that in an exchange ring $R$, every von Neumann regular element is unit regular if and only if $R$ has stable range one. And, from above, a ring with stable range one is both left and right uniquely generated.  The following quick consequence of the above theorems shows us that  Kaplansky's   left (right) uniquely generated  property is symmetric for pseudo-morphic rings.  

\begin{corollary}
Let $R$ be a ring. If $R$ is  pseudo-morphic, then the following are equivalent :
\begin{enumerate}
\item $R$ is left uniquely generated.\\
\item $R$ has stable range one. \\
\item $R$ is right uniquely generated.\\
\end{enumerate}
\end{corollary}

\bigskip
\bigskip
\section*{Acknowledgement}
\bigskip
\noindent  I would like to thank  Professor Ashish K Srivastava for his help and  encouragement with this work.

\bigskip

\bigskip

\end{document}